\documentclass[11pt]{amsart}
\usepackage[english]{babel}
\usepackage{verbatim}
\usepackage{amsmath,latexsym,amssymb,verbatim}
\usepackage{graphicx}
\usepackage{color}
\usepackage{caption}
\usepackage{subcaption}

\newtheorem{lema}{Lemma}[section]
\newtheorem{theo}[lema]{Theorem}
\newcounter{teoremaganso}

\newtheorem{prop}[lema]{Proposition}

\newtheorem{coro}[lema]{Corollary}
\theoremstyle{definition}

\theoremstyle{remark}

\newtheorem{exam}[lema]{Example}

\title[Abel equations with invariant curves]{Upper Bounds of Limit Cycles in Abel Differential Equations with Invariant Curves}

\author{J.L. Bravo}
\address{Departamento de Matematicas, Universidad de Extremadura, 06071 Badajoz, Spain}
\email{trinidad@unex.es}
\author{L.A. Calder\'on*}
\address{Departamento de Matematicas, Universidad de Extremadura, 06071 Badajoz, Spain}
\email{lucalderonp@unex.es}
\author{M. Fern\'andez}
\address{Departamento de Matematicas, Universidad de Extremadura, 06071 Badajoz, Spain}
\email{ghierro@unex.es}
\thanks{* Corresponding author}
\thanks{
The authors are partially supported by Junta de Extremadura/FEDER grants numbers IB18023 and GR18023.
JLB and MF are also partially supported by MINECO/FEDER grant number MTM2017-83568-P}

\begin{document}

\begin{abstract}
New criteria are established for upper bounds on the number of limit cycles of periodic Abel differential equations
having two periodic invariant curves, one of them bounded.
The criteria are applied to obtain upper bounds of either zero or one limit cycle for planar differential systems.
\end{abstract}

\maketitle

\section{Introduction and Statements of Main Results}

Bounding the number of periodic solutions of the generalized Abel equation
\begin{equation}\label{eq:genAbel}
	x'=\sum_{k=0}^n C_k(t) x^k,\qquad (t,x)\in\mathbb{R}^2,
\end{equation}
where $C_k(t)$ are $T$-periodic smooth functions, is an open problem known as the Smale-Pugh problem \cite{S}.

We shall denote by $u(t,x)$ the solution of Equation~\eqref{eq:genAbel} determined by $u(0,x)=x$. Recall that a solution $u(t,x)$  is {\it closed} or {\it periodic} if
$u(T,x)=x$, and that a {\it limit cycle} is a periodic solution which is isolated in the set of all periodic solutions.

When $n=1$, Equation~\eqref{eq:genAbel} is a linear
equation, and it is known that there is at most one limit
cycle.  When $n=2$, Equation~\eqref{eq:genAbel} is the
well-known Ricatti equation, and has at most two limit
cycles (see for example \cite{LN}). However, when $n=3$, A.~Lins-Neto~\cite{LN} proved that there exists no upper bound
on the number of limit cycles when $C_k$ belongs to the
whole family of trigonometric polynomials.  At this point,
in order to obtain upper bounds on the number of limit
cycles of Equation~\eqref{eq:genAbel}, one may impose some
additional constraints on its coefficients, $C_k(t)$ (see for
example
\cite{ABF2,ABF4,AGG,BF13,GG,GL,HL1,HL4,HL5,HZ,Ll,P,Pliss}).

In this work, we study \eqref{eq:genAbel} for $n=3$ with two invariant periodic curves of the form $a_0(t)x=b_0(t)$ and $a_1(t)x=b_1(t)$, where $a_0,b_0,a_1,b_1$ are smooth $T$-periodic functions, $a_0(t),a_1(t)\not\equiv 0$, and the first invariant curve is bounded, i.e., $a_0(t)\neq0$ for all $t\in[0,T]$.
By the change of variables $x\to x-b_0(t)/a_0(t)$, we may assume the first invariant curve to be $x=0$.
Therefore Equation \eqref{eq:genAbel} can be written as
$$
x'=x(C_1(t)+C_2(t)x + C_3(t)x^2).
$$ 	
For the second invariant curve, we shall assume that
$b_1(t)\equiv 1$, and that the set on which $a_1$ vanishes has
null measure.  If the functions $a_1,b_1$ are regular
enough, these assumptions are automatically satisfied. For
instance, if they are analytic, $a_1,b_1$ have a finite
number of zeros in $[0,T]$ with bounded multiplicity, and,
by uniqueness of solutions, the invariant curves are
disjoint, so that every zero of $b_1$ is a zero of $a_1$, and
this zero of $a_1$ has multiplicity greater than or equal to that of
the zero of $b_1$. Hence, dividing by $b_1(t)$ if necessary, one
may assume that $b_1(t)\equiv 1$.

Dividing the quadratic polynomial, $C_1(t)+C_2(t)x + C_3(t)x^2$ by $a_1(t)x-1$, one gets (whenever $a_1(t)\neq 0$)
\begin{equation}\label{eq:Abelinv}
x'=x\Big((a_1(t)x-1)(a_2(t)x-b_2(t))+ a_3(t)\Big),
\end{equation}
for certain $T$-periodic smooth functions $a_2,b_2,a_3$ integrable in $[0,T]$, and possibly with poles at the zeros
of $a_1$.

Imposing that $a_1(t)x-1=0$ is invariant, one obtains that, for any $t$ such that $a_1(t)\neq 0$, \eqref{eq:Abelinv}
reduces to
\begin{equation}\label{ode:1}
\begin{split}	
x'=&(a_1x-1)(a_2x-b_2)x -\frac{a_1'}{a_1}x\\
	=&a_1a_2 x^3-\left(a_1b_2+a_2\right)x^2+\left(b_2 - \frac{a_1'}{a_1}\right)x,
\end{split}	
\end{equation}
where the arguments of the functions have been omitted for clarity.

If $a_2(t)$ is identically null, \eqref{ode:1} is a Riccati equation with at most one non-null limit cycle (see for example \cite{LN}).
So we shall assume that $a_2(t)$ is not identically null.

The invariant curves divide the plane into connected components.
Motivated by the applications to the planar systems, we may consider just two regions, gluing together the connected components separated by $x=\pm\infty$ (see Figure~\ref{U}).
The objective of this paper is to obtain criteria for  Equation \eqref{ode:1} to have either zero or at most one limit cycle with graph included in one of the regions, $U$. By the change of variables $x\to -x$, it is not restrictive to assume that $(t,x)\in U$ for $x>0$ small
enough.
\begin{figure}[]
\begin{subfigure}[b]{0.3\textwidth}
\begin{picture}(150,220)
\put(0,0){\includegraphics[width=110pt]{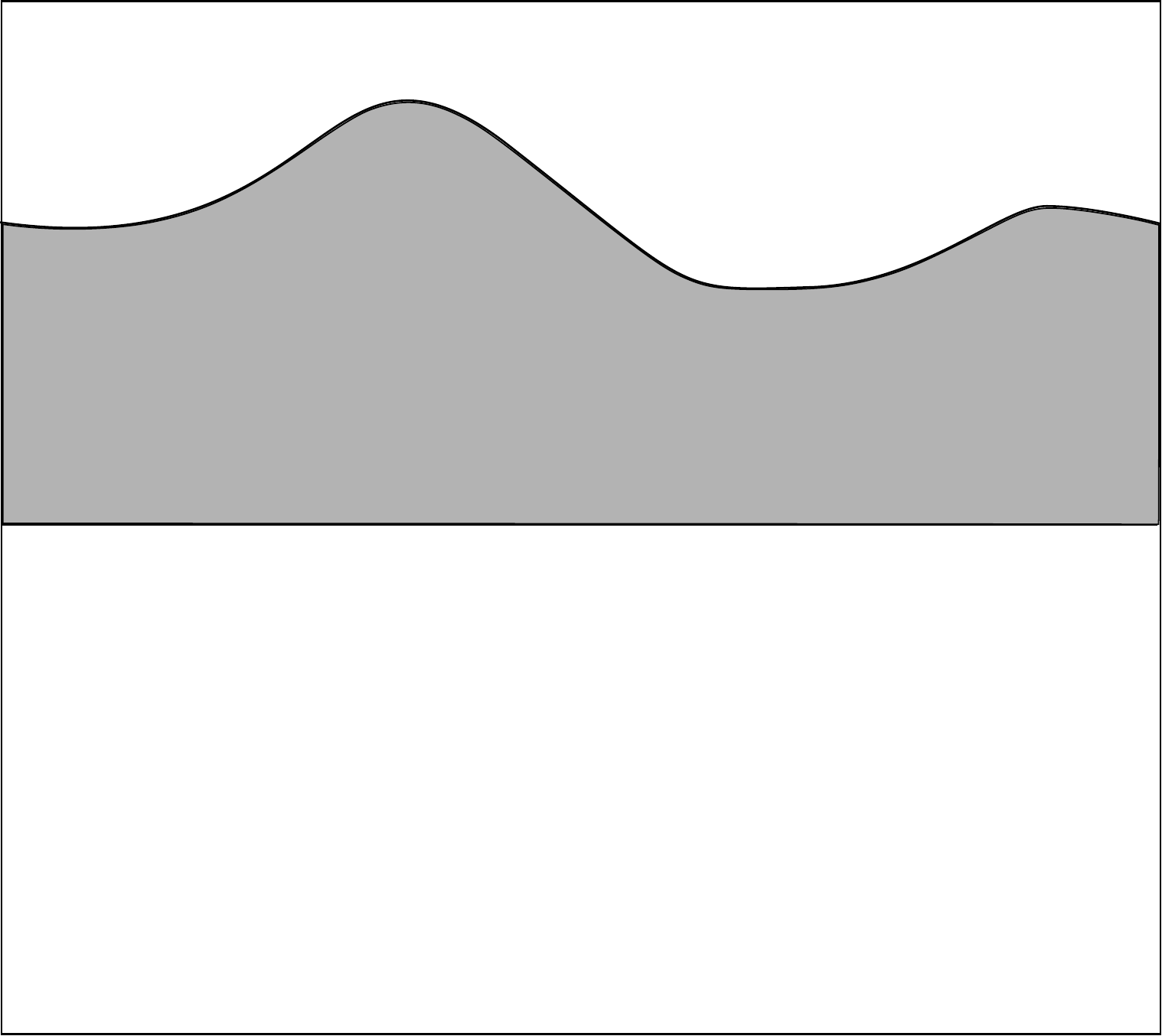}}
\put(90,60){U}
\put(100,40){t}
\end{picture}
\caption{$a_1>0$}
\end{subfigure}
\hfill
\begin{subfigure}[b]{0.3\textwidth}
\begin{picture}(150,220)
\put(0,0){\includegraphics[width=110pt]{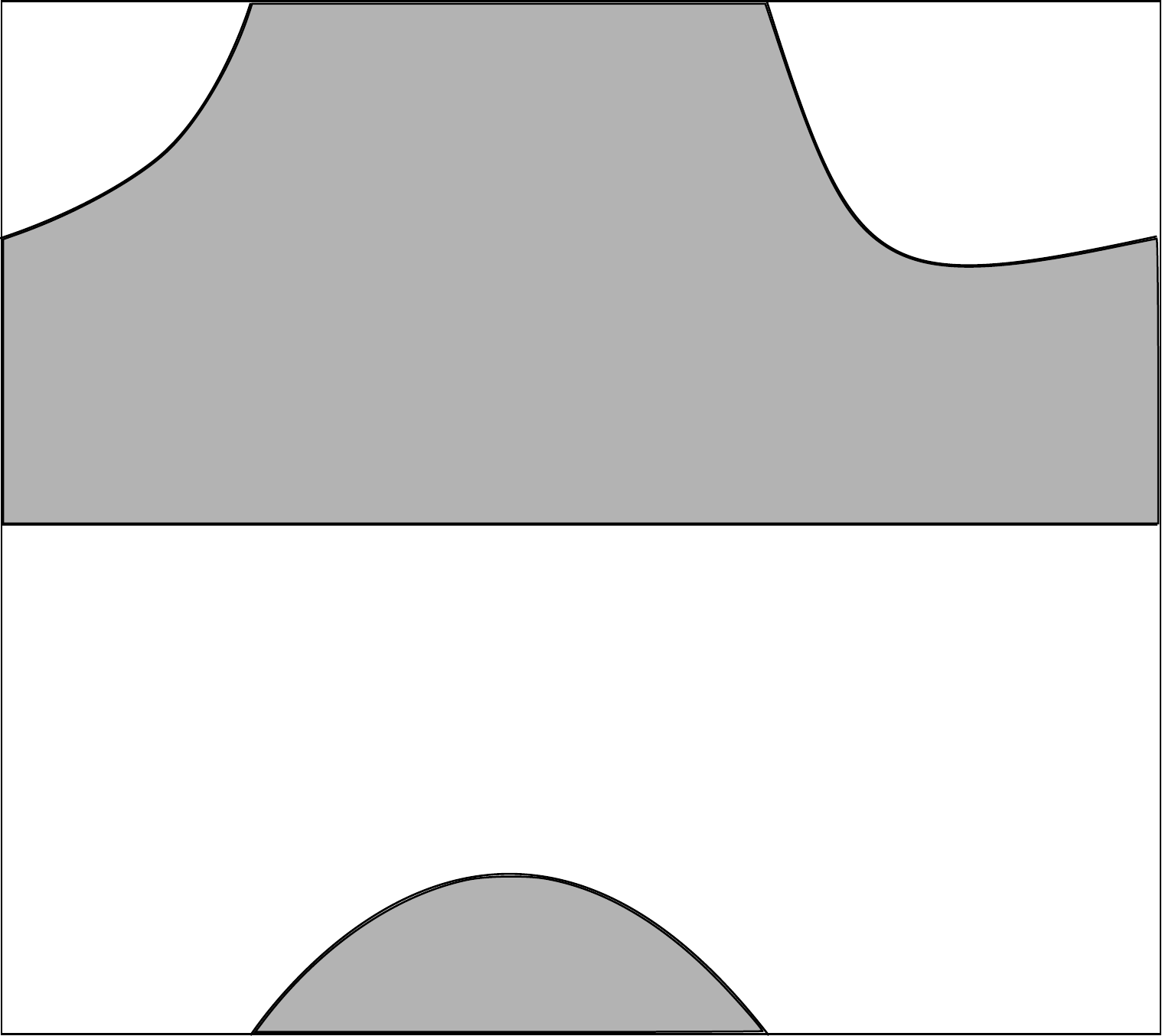}}
\put(90,60){U}
\put(100,40){t}
\end{picture}
\caption{$a_1$ changes sign}
\end{subfigure}
\hfill
\begin{subfigure}[b]{0.3\textwidth}
\begin{picture}(150,220)
\put(0,0){\includegraphics[width=110pt]{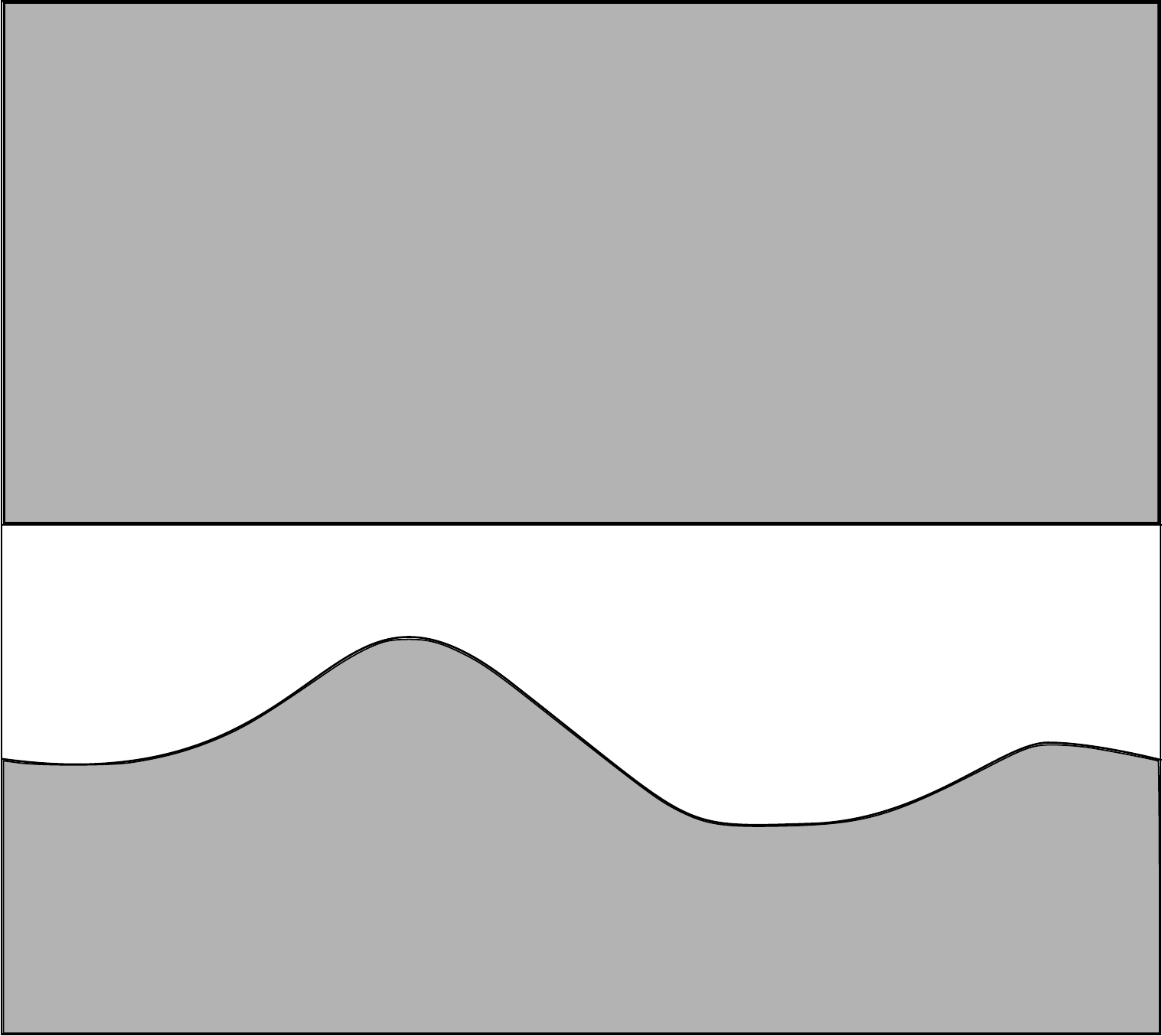}}
\put(90,60){U}
\put(100,40){t}
\end{picture}
\caption{$a_1<0$}
\end{subfigure}
\caption{Region $U$.}\label{U}
\end{figure}
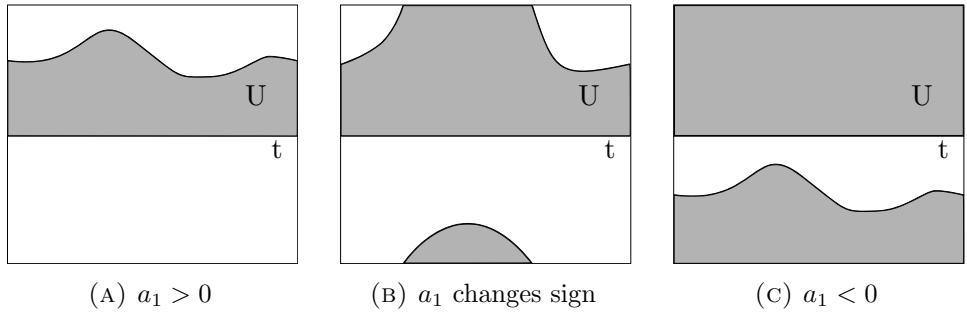

Indeed, as limit cycles are bounded, we may reduce the study to $V\subset U$ defined as:
\begin{itemize}
\item[(i)] when $a_1(t)>0$ for all $t$, \[V=\{ (t,x):  0<x<1/a_1(t)  \};\]

\item[(ii)] when $a_1$ has zeros,
\[
V=\{ (t,x): x>0 \text{ and } 0<x<1/a_1(t) \text{ for } t \text{ with }a_1(t) >0 \};\]

\item[(iii)] when $a_1(t)<0$ for all $t$, \[V= \{(t,x): x>0 \text{ or } x<1/a_1(t) \}.\]
\end{itemize}

Let $\bar b_1$ be a smooth $T$-periodic function such that $\bar b_1(t)\neq0$ for every $t$, and denote
\[\bar a_1(t)=a_1(t) \bar b_1(t),\quad \bar a_2(t)= \frac{a_2(t)}{\bar b_1(t)}, \quad \bar b_2(t)=\frac{b_2(t)+\frac{\bar a_1'(t)}{\bar a_1}}{\bar b_1(t)}.\]
Then Equation~\eqref{ode:1}
becomes
\begin{equation}\label{eq:HL}
x'=(\bar a_1x-\bar b_1)(\bar a_2x-\bar b_2)x + \frac{\bar b_1'-\bar a_1' x}{\bar b_1}x.
\end{equation}
Motivated by the study of limit cycles in planar systems,
this equation was proposed by Huang, Liang, Llibre in
\cite{HL5} who, by assuming that $\bar a_2,\bar b_2$ have no
poles, obtained the following criterion:
\begin{prop} \cite[Proposition 10]{HL5} \label{prop:hl}
	Equation \eqref{eq:HL} has at most one limit cycle  (counted with multiplicity) with graph included in $V$ if, for all $t \in [0,T]$, one of the following conditions holds:
\begin{itemize}
\item[(i)] $\bar a_1(t)\neq0$ and there exists $\eta\in\mathbb{R}$ such that $\bar a_1(t) \bar b_2(t)+\eta \bar a_2(t) \bar b_1(t)+ (\bar a_1'(t)/\bar b_1(t)) \geq0\ (\leq0)$;
\item[(ii)] $\bar a_2(t) \geq 0$ ($\leq 0$);
\item[(iii)] Either $\bar a_1(t) \bar a_2(t) \geq0$ and $\bar b_1(t)\bar b_2(t) \leq0$, or $\bar a_1(t) \bar a_2(t) \leq0$ and $\bar b_1(t)\bar b_2(t) \geq0$.	
\end{itemize}
\end{prop}

The objective of the paper is to establish new criteria for Equation \eqref{ode:1} to have either zero or at most one limit cycle  (counted with multiplicity) with graph in $V$.

We say a $T$-periodic function $f$ has positive definite sign if $f(t)\geq 0$ for all $t\in[0,T]$. Analogously,
$f$ has negative definite sign if $f(t)\leq 0$ for all $t\in[0,T]$ and $f$ has definite sign if it has either positive
or negative definite sign.

The main results are:
\begin{theo}\label{criterium1}
Assume there exists $\eta\in\mathbb{R}$ such that $b_2+\eta a_1'/a_1$ has positive (negative) definite sign  and for each  $t\in[0,T]$ the following conditions hold:
\begin{itemize}
	\item[(i)] if $a_1(t)< 0$  
	then $a_2(t)\leq 0$ ($\geq 0$);
	\item[(ii)] if $a_1(t)> 0$ then
	$a_1(t)b_2(t)-a_2(t)+\eta a_1'(t) \geq 0$, ($\leq0$);
\end{itemize}
with the inequalities in $(i)$ and $(ii)$ being strict for all $t\in P$, where
$P$ is a positive measure set of $[0,T]$. Then Equation \eqref{ode:1} has no non-trivial limit cycle  with graph included in $V$.
\end{theo}

\begin{theo}\label{criterium2}
Assume  there exists $\eta\in\mathbb{R}$ such that $b_2+\eta a_1'/a_1$ has positive (negative) definite sign, $b_2$ is not identically null, and for each  $t\in[0,T]$ the following conditions hold:
\begin{itemize}
	\item[(i)] if $a_1(t)<  0$,  then $a_2(t)\geq 0$ ($\leq 0$);
	\item[(ii)] if $a_1(t)> 0$ then
	$a_1(t)b_2(t)-a_2(t)+\eta a_1'(t) \geq 0$, ($\leq0$);
\end{itemize}with the inequalities in $(i)$ and $(ii)$ being strict for all $t\in P$, where
$P$ is a positive measure set of $[0,T]$.
Then Equation \eqref{ode:1} has at most one non-trivial limit cycle  with graph included in $V$.
Moreover,  this upper bound is attained.
\end{theo}

Note that  $a_1(t)$ and $a_2(t)$ can have any number of zeros. Moreover, since $a_2,b_2$ may also have poles, $\eta$ must be chosen such that $b_2+\eta a_1'/a_1$ has only poles of even order.

An important application of the results above is to bound the number of limit cycles of polynomial planar systems. Indeed, many
planar systems reduce to an Abel equation after a change of variables (see~\cite{DLP,Lloyd} for example). In these systems, the origin is a singular point that is transformed into the periodic solution $x=0$, so that, if a second invariant curve is known, the equation is of the type
\eqref{ode:1} and Theorems \ref{criterium1} and \ref{criterium2} apply.

In Section 4, we shall detail this for homogeneous planar systems studied in \cite{CLl,HL5} and compare 
the previous known criteria with the new ones obtained for this case.

\section{Preliminaries}\label{sec:2}

Consider the equation
\begin{equation}\label{eq:esc}
x'=p(t,x),
\end{equation}
where $p \colon \mathbb{R}^2 \to \mathbb{R}$ is sufficiently smooth, $T$-periodic with respect to $t$, i.e., $p(t,x)=p(t+T,x)$, and polynomial in $x$. 
Let $A=\{x\in\mathbb{R}: u(T,x) \text{ is defined}\}$.
Denote $d(x)=u(T,x)-x$, the displacement application. Then,
$x_0$ is a zero of $d(x)$ if and only if $u(t,x_0)$ is a periodic solution of \eqref{eq:esc}. So  isolated zeros of $d(x)$ are  initial
conditions of  limit cycles of \eqref{eq:esc}.

Deriving \eqref{eq:esc} with respect to the initial condition, one obtains
\begin{equation}\label{eq:d'}
d'(x)=\exp\left(\int_{0}^{T}p_x(t,u(t,x))\,dt\right)-1.
\end{equation}
When  $d(x)=0$ and $d'(x)\neq0$, the limit cycle $u(t,x)$ is called hyperbolic. In this case the sign of $d'(x)$ determines its stability.

Using elementary analysis, it is easy to prove the following result:
\begin{lema}
Let $I\subset A$ be an interval, and assume that $d'(x)> 0$ (resp. $<0$)  for all $x\in I$
such that $u(t,x)$ is periodic. Then there is at most one periodic solution with initial
condition in $I$.
\end{lema}

Let $q(t,x)$ be a polynomial in $x$ with coefficients that are smooth, $T$-periodic with respect to $t$, functions.
In this context, we say that $q(t,x)=0$ is a periodic algebraic invariant curve if there exists  $Q(t,x)$ which is $T$-periodic with respect to $t$ and polynomial in $x$, called the cofactor of $q(t,x)$, such that
\[
q_t(t,x)+q_x(t,x)p(t,x)=q(t,x) Q(t,x).
\]
Note that the coefficients of $Q$ as a polynomial in $x$ are smooth functions that might have poles at the zeros of the coefficients of $q(t,x)$.

Let $u(t,x)$ be a periodic solution of \eqref{eq:esc}. By
uniqueness of solutions, either 
$q(t,u(t,x))\neq 0$ for all $t$ or $q(t,u(t,x))\equiv 0$. Assume that the first possibility is the
case and that the set of $t$ such that the coefficients of
$Q(t,x)$ have poles has null measure.  Then
\[
\int_0^T Q(t,u) \, dt=\int_0^T \frac{ q_t(t,u)+q_x(t,u)p(t,u)}{q(t,u)}\,dt=\int_0^T \frac{d\left(\ln|q(t,u)|\right)}{dt}\,dt=0,
\]
where $u=u(t,x)$.
Therefore, for any $\alpha\in\mathbb{R}$, the sign of $d'(x)$ is the sign of
\[
\int_{0}^{T} \left(p_x(t,u) + \alpha Q(t,u) \right) \,dt.
\]

Now consider \eqref{ode:1}. In this case,
\[
p_x(t,x)=3a_1a_2x^2-2\left(a_1b_2+a_2\right)x+b_2-\frac{a_1'}{a_1},
\]
and the cofactors of $x=0$ and $a_1x-1=0$ are
\[
p_1(t,x)=(a_1 x - 1) (a_2 x - b_2) -\frac{a_1'}{a_1}\]
and
\[
p_2(t,x)= a_1 x (a_2 x - b_2).
\]

Moreover, since $a_1$ is $T$-periodic, then
\[
\int_0^T \frac{a_1'}{a_1}\,dt=0.
\]
For any $\alpha,\beta,\eta\in\mathbb{R}$, denote
\begin{equation}\label{eq:G}
\begin{split}
G(t,x)=& p_x(t,x)+\alpha p_1(t,x)+\beta p_2(t,x) +(1+\alpha + \eta) \frac{a'_1(t)}{a_1(t)} \\
&=(3 + \alpha + \beta) a_1 a_2 x^2  -  \left((2 + \alpha) a_2 + (2 + \alpha + \beta) a_1 b_2\right) x \\
&+  (1 + \alpha)b_2 +\eta\frac{a_1'}{a_1}.
\end{split}
\end{equation}

A direct consequence is the following result: \begin{lema}\label{lema:principal}
Assume there exist $\alpha,\beta,\eta\in\mathbb{R}$ such that for every $(t,x)\in V$ except for $t$ in
a zero-measure set
\[
G(t,x) \geq 0\quad \text{(resp. }\leq 0\text{)},
\]
and there exists a positive measure set $P$ such that for every $t\in P$ the inequality is strict.
Then there is at most one limit cycle in each connected component of $V$.
Moreover, if one exists then it is hyperbolic and unstable (resp. stable).
\end{lema}

To check when $G(t,x) \geq0$ (resp. $\leq0$) for each
$t\in[0,T]$, we shall use Sturm's theorem applied to $G$ as a
polynomial in $x$.  Recall that the Sturm sequence for $G$ when
$a_1(t),a_2(t)\neq0$ and $\alpha+\beta+3\neq 0$ is
$\{q_0(x),q_1(x),q_2\}$, where $q_0(x)=G(t,x)$,
$q_1(x)=G_x(t,x)$, and
\[
q_2=\frac{(2 + \alpha)^2 a_2 }{4 (3 + \alpha + \beta) a_1}
 +\frac{ (2 + \alpha + \beta)^2 a_1 b_2^2}{4 (3 + \alpha + \beta)  a_2}
-\frac{ (2 + \alpha^2 + \alpha (4 + \beta)) b_2}{2 (3 + \alpha + \beta)}+\eta \frac{a_1'}{a_1}.
\]
Let $a\in\mathbb{R}$ with $q_0(a)\neq0$. Define $v(a)$ as the number of changes of sign in the sequence $\{q_0(a), q_1(a), q_2\}$, omitting the zeros.

\begin{theo} [{Sturm's theorem, see for example \cite[p. 297]{BS}}]
	If  $a<b$,  $q_0(a)\neq0$, and $q_0(b)\neq0$, then the number of different zeros (counting the multiple roots just once) of $q_0(x)$ in $[a,b]$ is $v(a)-v(b)$.
\end{theo}
Thus, fixed $t\in[0,T]$, $a,b\in\mathbb{R}$ such that $G(t,a)$ and $G(t,b)$ differ from $0$, if $v(a)=v(b)$, then  $G(t,x)$ has definite sign in the interval $(a,b)$.
Note that Sturm's theorem is also valid for $a=-\infty$ or $b=+\infty$.

\section{Proofs of the Main Results}

The key to the proof of Theorems \ref{criterium1} and
\ref{criterium2} is to choose special values for
$\alpha,\beta, \eta$ to simplify~\eqref{eq:G}, and then to
apply Sturm's theorem so as to obtain the conditions for $G\geq 0$
($G\leq 0$) in $V$, when necessary.

In the case of $a_1(t)<0$ for all $t\in[0,T]$, our results
imply a certain definite sign of $a_2$. But a more general
criterion can be stated. This criterion was already obtained
in \cite{HL3}, but we include it here for completeness.

\begin{prop}[{\cite[Lemma 2.3]{HL3}}] \label{prop:signo_definido}
Assume that 
 $a_2$ has  definite sign in $[0,T]$, strict for some positive measure set.
Then Equation \eqref{ode:1} has at most one non-trivial limit cycle  with graph included in $V$.
\end{prop}
\begin{proof}

Firstly, take $\alpha=\beta=-1$, and $\eta=0$. Then
\[G(t,x)=a_2(t)(a_1(t)x-1)x.\]
Therefore, if $a_2$ has definite sign, strict for some positive measure set, the same is the case for $G(t,x)$,
and there is at most one limit cycle with graph
in each of the connected components of $V$.

Thus it suffices to prove that if $V$ has two connected components, i.e., if
$a_1(t)<0$ for all $t$, then there is at most one limit cycle with graph in $V$.

The transformation $y=a_1(t)x$ reduces \eqref{ode:1} to
\begin{equation}\label{eq:y}
y'=\frac{1}{a_1(t)}y(y-1)\left(a_2(t)y-a_1(t)b_2(t)\right),
\end{equation}
and $V$ to $\tilde{V}=\{(t,y):y<0 \text{ or } y>1\}$.

If we assume that $a_2$ has definite sign, strict in some positive measure set, then the result follows comparing the solutions of \eqref{eq:y}
with the solutions of the separated variables equation $y'=b_2(t) y(y-1)$. See \cite[Lemma 2.3]{HL3} for the details.
\end{proof}

Now let us prove the first theorem.

\begin{proof}[Proof of Theorem \ref{criterium1}]
    Firstly, if $a_1(t)<0$ for every $t$ then, by
  Proposition~\ref{prop:signo_definido}, we conclude there
  is at most one positive limit cycle with graph in $V$. So
  we assume $a_1$ has zeros or $a_1>0$.

Choosing $\eta$ conveniently and $t$ such that $a_1(t)\neq 0$, we may write
\[
\begin{split}
G(t,x)=& (3 + \alpha + \beta) a_1(t) a_2(t) x^2  -  \left((2 + \alpha) a_2(t) + (2 + \alpha + \beta) a_1(t) b_2(t)\right) x
\\&+  (1 + \alpha)\left( b_2(t) +\eta\frac{a_1'(t)}{a_1(t)}\right).
\end{split}
\]

Consider~\eqref{eq:G} with $\alpha + \beta+2=0$. Then, for any $t$ such that $a_1(t)\neq0$,
\[
G(t,x)= a_1(t)a_2(t) x^2-(\alpha+2) a_2(t)x +(\alpha+1)\left( b_2(t) +\eta\frac{a_1'(t)}{a_1(t)}\right).
\]
Assume $b_2(t)+\eta a_1'(t)/a_1(t) \geq0$ for all $t\in[0,T]$, with the case $\leq0$ being analogous. Fix $\alpha>0$.
We shall prove that $G(t,x) \geq 0$ for every $(t,x)\in V$ such that $a_1(t)\neq0$, with strict inequality for every $t$ in some positive measure set.

Fix any $t\in [0,T]$ such that $a_1(t)\neq 0$. First, we shall consider the singular cases.

\textsc{Case 1}. Assume $a_2(t)=0$. Then trivially $G(t,x)\geq 0$ for every $x\geq 0$.

\textsc{Case 2}. Assume $b_2(t)+\eta a_1'(t)/a_1(t)=0$ and $a_2(t)\neq 0$. Then
\[
G(t,x)=a_1(t)a_2(t)x^2-(\alpha+2)a_2(t)x.
\]
If $a_1(t)<0$ then, by hypothesis, $a_2(t)\leq 0$. Then
$a_1(t)a_2(t)> 0$, and $G(t,x)\geq 0$ for every $x \geq 0$.
If $a_1(t)>0$, we only need to check that $G(t,x)\geq 0$ for
all $x\in (0,1/a_1(t))$.  By hypothesis,
$a_1(t)b_2(t)-a_2(t)+\eta a_1'(t)\geq 0$, and since
$a_1(t)b_2(t)+\eta a_1'(t)=0$ then $a_2(t)< 0$.  As the
function can be rewritten as
\[G(t,x)=a_2(t)(a_1(t)x-(\alpha+2))x,\]
and $x\geq 0$, we only need to check that the linear
function $G^*(t,x)=a_2(t)(a_1(t)x-(\alpha +2))$ is positive in the interval, so we must verify that it is
positive at the extrema of the interval $(0,1/a_1(t))$. But
\[
G^*(t,0)=-a_2(t)(\alpha+2)> 0,\quad
G^*(t,1/a_1(t))=a_2(t)\left(1-(\alpha+2) \right)> 0.
\]

For the remaining cases, consider the Sturm sequence
\[S(x)=\{q_0(x),q_1(x),q_2\},\]
where
\begin{align*}
q_0(x)&=G(t,x),\\
q_1(x)&=G_x(t,x)= 2 a_1(t)a_2(t)x-(\alpha+2)a_2(t),\\
q_2\ \ &= \frac{(\alpha+2)^2a_2(t)}{4 a_1(t)}+\frac{(\alpha+2)(\alpha+\beta+2)}{2(\alpha+\beta+3)}b_2(t) - (\alpha+1)\left(b_2(t)+\eta \frac{a_1'(t)}{a_1(t)} \right).
\end{align*}

\textsc{Case 3}. Let $t$ be such that $a_1(t)< 0$ and $a_2(t)(b_2(t)+\eta a_1'(t)/a_1(t)) \neq 0$. 
The Sturm sequences at $x=0$ and  $x=\infty$ are
\begin{align*}
&S(0)=\{(\alpha+1)(b_2(t)+\eta a_1'(t)/a_1(t)),-(\alpha+2)a_2(t),q_2 \}, \\
&S(\infty)=\{a_1(t)a_2(t),2a_1(t)a_2(t),q_2 \}.
\end{align*}
By hypothesis, $a_2(t)< 0$. Then
$a_1(t)a_2(t)> 0$. Therefore, the Sturm sequences $S(0)$ and $S(\infty)$ have the same changes of sign as the sequence
\[\{ 1,1,q_2\}.\]
Hence, $G(t,x)$ has no  zeros in $(0,\infty)$.
 As $G(t,0)> 0$, one has that $G(t,x)> 0$ for all
  $x\in(0,\infty)$.

\textsc{Case 4}. Let $t$ be such that $a_1(t)> 0$ and $a_2(t)(b_2(t)+\eta a_1'(t)/a_1(t)) \neq 0$. 

The Sturm sequences at $x=0$ and $x=1/a_1(t)$ are
\begin{align*}
& S(0)=\{(\alpha+1)(b_2(t)+\eta a_1'(t)/a_1(t)),-(\alpha+2)a_2(t),q_2 \}, \\
& S\left(\frac{1}{a_1(t)}\right)=\left\{\frac{(1 + \alpha) (
   a_1(t) b_2(t) -a_2(t) + \eta a_1')}{a_1(t)}, -\alpha a_2(t),q_2 \right\}.
\end{align*}
Since $b_2(t)+\eta a_1'(t)/a_1(t)>0$ and, by hypothesis, $a_1(t)b_2(t)-a_2(t)+\eta a_1'(t)\geq0$, if the inequality is strict then the Sturm sequences $S(0)$ and
$S(1/a_1(t))$ have the same signs as the sequence
\[
\{1,-a_2(t),q_2\}.
\]
Hence $G(t,x)$ has no zeros in the interval $(0,1/a_1(t))$.  As $G(t,0)> 0$, one has that $G(t,x)>0$ for all $x\in(0,1/a_1(t))$.

If $a_1(t)b_2(t)-a_2(t)+\eta a_1'(t)=0$, then $a_2(t)>0$ and $G(t,x)= a_2(t)(a_1(t)x^2-(\alpha+2)x+(\alpha+1)/a_1(t))>0$ for each $x \in (0,1/a_1(t))$, because the function $g(t,x)= G(t,x)/a_2(t)$ satisfies $g(t,0)>0, g(t,1/a_1(t))=0$, and $g_x(t,x)<0$ for each $x \in (0,1/a_1(t))$.

So far we have obtained that there is at most one
non-trivial limit cycle with graph included in $V$. Now we
will show that in this case there is no such limit cycle.
Assume that $b_2(t)+\eta a_1'(t)/a_1(t)$ differs from
zero on a positive measure set. To prove that there is no
limit cycle with graph in $V$, note that
\[
d'(0)= \exp \left(\int_{0}^{T}b_2(t)+\eta \frac{a_1'(t)}{a_1(t)}dt\right)-1 >0,
\]
so that the origin is unstable. By Lemma~\ref{lema:principal}, if there exists a limit cycle with graph in $V$ then
it is unstable, so that there is no limit cycle in the connected component of $V$ contained in $x\geq 0$.

If $b_2(t)+\eta a_1'(t)/a_1(t)\equiv 0$ then $d'(0)=0$ and
$$
d''(0)=-2\int_{0}^{T} a_2(t) \,dt >0.
$$
Hence, the null solution is unstable, so that there is no such limit cycle in the connected component of $V$ contained in $x\geq 0$.

Assume now that $a_1(t)<0$ for every $t$. In order to prove that there is no limit cycle contained in the connected
component of $V$ contained in $x\leq 0$, note that the stability of the limit cycle $x(t)=1/a_1(t)$ is
\[
d'(1/a_1(0))= \exp\left(\int_0^T \left(a_1(t)b_2(t)-a_2(t) + \eta a_1'(t)/a_1(t)\right) \,dt \right)-1 >0,
\]
so that the cycle $x(t)$ is unstable. Again by Lemma~\ref{lema:principal}, if there exists a limit cycle with graph in $V$ then
it is unstable, so that there is no  limit cycle in the connected component of $V$ contained in $x\leq 0$.
\end{proof}

Now, let us prove the second main result.

\begin{proof}[Proof of Theorem \ref{criterium2}]
	If $a_1(t)\leq0$ for every $t$, then $a_2$ has definite sign,   and by Proposition~\ref{prop:signo_definido} there is at most one limit cycle with graph contained in $V$.  So we may assume that either $a_1(t)>0$ for every $t$ or $a_1$ has zeros.

Consider~\eqref{eq:G} with $\alpha=-2$, $\beta=0$, and $\eta$ such that
\[
G(t,x)= a_1(t)a_2(t) x^2 -\left(b_2(t)+\eta\frac{a_1'(t)}{a_1(t)}\right).
\]
We shall prove that $G(t,x)$ has definite sign in $V$ for every $t$ such that $a_1(t)\neq 0$.

Firstly, assume that $b_2(t)+\eta a_1'(t)/a_1(t) \geq0$, with the other case  being analogous.
We shall prove that $G(t,x) \leq 0$ for all $(t,x)\in V$.

Fix $t\in [0,T]$. If $a_1(t)< 0$ then, by hypothesis, $a_2(t)\geq 0$, so that $a_1(t)a_2(t)\leq 0$ and $G(t,x)\leq 0$ for all $x\in [0,+\infty)$.

Now, if $a_1(t)> 0$, as the only extreme of $G(t,x)$ is at $x=0$, and $G(t,0)\leq 0$, we only need to prove that $G(t,1/a_1(t))\leq 0$. But, if $b_2(t)\neq 0$ then
\[
G\left(t,\frac{1}{a_1(t)}\right)=\frac{a_2(t)}{a_1(t)} - \left(b_2(t)+\eta\frac{a_1'(t)}{a_1(t)}\right)=\frac{a_2(t)-b_2(t)a_1(t)-\eta a_1'(t)}{a_1(t)}\leq 0.
\]

Finally, to prove that this upper bound is attained, it suffices to take $a_1,a_2,b_2\in\mathbb{R}$,
$a_1>0$, $0<b_2/a_2<a_1$.
\end{proof}

\section{Planar Polynomial Differential Systems with Homogeneous Singularities}

In this section, we show some examples of applications of the results to planar systems.

\subsection*{Rigid systems}
A family that is easily transformed into Abel equations (or into generalized Abel equations) is that of
 rigid systems, i.e., systems of the form
\begin{equation}\label{eq:rigid}
\begin{cases}
x'=-y+x p(x,y),\\
y'=\phantom{-}x+yp(x,y),
\end{cases}
\end{equation}
where $p$ is a polynomial.

These systems have been studied in~\cite{GPT,GT} for instance.
After a change to polar coordinates, \eqref{eq:rigid} becomes
\begin{equation}\label{eq:rigidabel}
r'=rp(r\cos\theta,r\sin\theta),\quad \theta'=1.
\end{equation}
Moreover, there is a correspondence between limit cycles of \eqref{eq:rigid} and positive limit cycles of~\eqref{eq:rigidabel}.

In order to have an Abel equation (though the methods of Section~\ref{sec:2} could be applied for any degree) with sufficient degrees
of freedom to apply Theorem~\ref{criterium2}, let $p$ be of the form
\[
p(x,y)=p_{00}+\sum_{i=0}^k p_{i,k-i} x^i y^{k-i}+\sum_{i=0}^{2k} p_{i,2k-i} x^i y^{2k-i}.
\]
After the change of variables $\rho=r^k$, \eqref{eq:rigidabel} becomes the Abel equation
\begin{equation}\label{eq:rigidrho}
\begin{split}
\rho'=&k p_{00}\rho+k\left(\sum_{i=0}^k p_{i,k-i} \cos^i \theta \sin^{k-i} \theta\right)\rho^2\\
&+ k \left(\sum_{i=0}^{2k} p_{i,2k-i} \cos^i \theta \sin^{2k-i} \theta \right)\rho^3.
\end{split}
\end{equation}
If \eqref{eq:rigidrho} has an invariant curve with poles then the study of positive limit cycles
can be reduced to the study of limit cycles with graph in $V$.

\begin{exam}
Consider~\eqref{eq:rigid} with
\[
p(x,y)=1 - \frac{1}{2} x^4 y^2 + x^3 y^3 - \frac{5}{2} x^2 y^4 + x y^5 - 2 x^6 y^6 +
 3 x^5 y^7 - x^4 y^8.
\]
It is easy to check that $xy=1$ is an invariant curve for the planar system. Moreover, 
$x^3y^3=1$ is also an invariant curve. Then, after the change to polar coordinates and the
change $\rho=r^6$, $\rho \cos^3\theta\sin^3\theta=1$ is an invariant curve of \eqref{eq:rigidrho}.
Dividing by $a_1(\theta) \rho -1$, where $a_1(\theta)=\cos^3\theta\sin^3\theta$,
one obtains
\[
a_2(\theta)= -12 \cos^3 \theta \sin^3\theta + 18 \cos^2 \theta \sin^4\theta - 6 \cos \theta \sin^5\theta,\]\[
b_2(\theta)=6 + 3 \cot \theta  - 3 \tan \theta.
\]

Now let us check that Theorem~\ref{criterium2} is satisfied. In order for $b_2+\eta a_1'/a_1$ to have definite sign,
we must choose $\eta=-1$, obtaining $b_2+\eta a_1'/a_1=6$.

To verify that (i) and (ii) are satisfied, note that $a_1$
and $a_2$ are $\pi$-periodic, so that it suffices to show
that they are satisfied in $(-\pi/2,\pi/2)$.  Dividing by
$\cos^6\theta$, one obtains
\[
\frac{a_1(\theta)}{\cos^6 \theta}=\tan^3 \theta,\quad
\frac{a_2(\theta)}{\cos^6 \theta}=-12 \tan^3\theta + 18 \tan^4\theta - 6 \tan^5\theta.
\]
Now we define the auxiliary polynomials
\[
p_1(t)=\frac{a_1(\arctan t)}{\cos^6 \arctan t}=t^3,\quad
p_2(t)=\frac{a_2(\arctan t)}{\cos^6 \arctan t}=-12 t^3 + 18 t^4 - 6 t^5,\]\[
p_3(t)=\frac{a_1(\arctan t)b_2(\arctan t)-a_2(\arctan t)-a_1'(\arctan t)}{\cos^6 \arctan t}=
6 t^3 (3 - 3 t + t^2).
\]
Note that the signs of $a_1,a_2,a_1b_2-a_2+\eta a_1'$ in $\theta\in(-\pi/2,\pi/2)$
are the signs of $p_1,p_2,p_3$ in $t=\tan \theta\in \mathbb{R}$, respectively.
Therefore, as $p_1(t)$ is negative for $t<0$ and positive for $t>0$,
it only remains to check that $p_2(t)>0$ for $t<0$ and that $p_3(t)>0$ for $t>0$.
But this is immediate as $p_2$ has the roots $0,1,2$ and $p_3$ has only the root $0$ (and two complex roots).

Moreover, as $a_1$, $a_2$, and $a_1a_2$ have changes of sign, the criteria in Proposition~\ref{prop:hl}
are not satisfied.

\end{exam}

\subsection*{Homogeneous systems}
Consider the homogeneous planar system studied in \cite{CLl,HL5}:
\begin{align}\label{sistema plano}
x'&= a x -y	+P_n(x,y),\\
y'&= x+a y +Q_n(x,y), \notag
\end{align}	
where $a\in\mathbb{R}$ and $P_n,Q_n$ are homogeneous polynomials of degree $n$.
In polar coordinates, \eqref{sistema plano} is written as
\begin{align*}
r'&= a r+ \varphi (\theta)r^n,\\
\theta'&=1+ \psi(\theta)r^{n-1},\notag	
\end{align*}	
where
\begin{align*}
\varphi(\theta)&=P_n(\cos\theta,\sin\theta) \cos\theta+Q_n(\cos\theta,\sin\theta)\sin\theta,\\
\psi(\theta)&= Q_n(\cos\theta,\sin\theta)\cos\theta-P_n(\cos\theta,\sin\theta)\sin\theta.	
\end{align*}
Note that $\psi(\theta)$ has a finite number of zeros in $[0,2\pi]$.	
Since limit cycles of \eqref{sistema plano} surrounding the origin do not intersect the curve in the $(r, \theta)$ plane, $1+ \psi(\theta)r^{n-1}=0$
(see~\cite{CLl}), they can be determined through the  limit cycles of the scalar equation
\begin{equation}\label{ode:1polar}
	\frac{dr}{d\theta}= \frac{a r+ \varphi (\theta)r^n}{1+ \psi(\theta)r^{n-1}}, \quad \theta \in \mathbb{R}.
\end{equation}

We may assume that $\psi(\theta)$ is not identically null since otherwise \eqref{ode:1polar} would be a Ricatti equation that has at most one non-null limit cycle.
	
Now, using Cherkas's change of variable \cite{Che}
\begin{equation*}
\rho= \frac{r^{n-1}}{1+ \psi(\theta)r^{n-1}},	
\end{equation*}	
Equation \eqref{ode:1polar} transforms into the Abel equation
\begin{equation}\label{ode:1cherkas}
\begin{split}\frac{d\rho}{d\theta}&=(\psi \rho -1)((n-1)(a\psi-\varphi)\rho-(n-1)a)\rho-\psi' \rho^2 \\
&=(n-1)(a\psi-\varphi)\psi\rho^3+((n-1)(\varphi-2a\psi)-\psi')\rho^2+ (n-1)a \rho,
\end{split}	
\end{equation}	
with invariant curves $\rho=0$ and $\psi(\theta)\rho-1=0$.

Moreover, as the change of variable is equivalent to
\[
r^{n-1}=\frac{\rho}{1-\psi(\theta)\rho},
\]
the set $\{(\theta,r) : \theta \in\mathbb{R}, r>0 \}$ is transformed into the region $V$ defined above (see~\cite{CLl,GL} for more details).

Therefore, Equation \eqref{ode:1cherkas} is of the form \eqref{ode:1} with
\begin{equation*}
a_1= \psi, \quad a_2= (n-1)(a\psi-\varphi),  \quad b_2= (n-1)a+\frac{\psi'}{\psi}.	
\end{equation*}	

In order to apply Theorems~\ref{criterium1} and \ref{criterium2}, it is necessary that
\[
b_2+\eta a_1'/a_1=(n-1)a-(1+\eta)\frac{\psi'}{\psi}
\]
has definite sign. To allow $\psi$ to have zeros, we choose $\eta=-1$.

Now, Theorems \ref{criterium1} and \ref{criterium2} have the following consequences, respectively:
\begin{coro}\label{coro:1}
Assume that $a> 0\ (<0)$ and that for each  $\theta \in[0,2\pi]$ the following conditions hold:
\begin{itemize}
	\item[(i)] if $\psi(\theta)< 0$  then $a\psi(\theta)-\varphi(\theta)\leq 0\ (\geq0)$,
	\item[(ii)] if $\psi(\theta)>0$ then $\varphi(\theta)\geq 0 \ (\leq0)$,
\end{itemize}
with the inequalities in (i) and (ii) being strict for every $\theta$ in some positive measure set.
Then Equation~\eqref{sistema plano} has no limit cycles surrounding the origin.
\end{coro}

\begin{coro} 
Assume that $a> 0\ (<0)$ and that for each  $\theta \in[0,2\pi]$ the following conditions hold:
\begin{itemize}
	\item[(i)] if $\psi(\theta)< 0$  then $a\psi(\theta)-\varphi(\theta)\geq 0\ (\leq0)$,
	\item[(ii)] if $\psi(\theta)>0$ then $\varphi(\theta)\geq 0 \ (\leq0)$,
\end{itemize}
with the inequalities in (i) and (ii) being strict for every $\theta$ in some positive measure set.
Then Equation~\eqref{sistema plano} has at most one limit cycle surrounding the origin.
\end{coro}	

We shall show an example of a planar polynomial differential
system with homogeneous cubic nonlinearities which satisfies
conditions (i) and (ii) in Corollary~\ref{coro:1} but
does not satisfy any of the conditions in \cite[Theorem 1,
Theorem 3, and Corollary 4]{HL5}.

To that end, consider \eqref{ode:1cherkas}.  A first comment
is that when the functions $\psi$ and $\varphi$ are
homogeneous trigonometric polynomials of odd degree then
conditions (i) and (ii) in Corollary~\ref{coro:1} imply that
$\varphi$ and $\psi$ have the same zeros. Indeed, in this
case, the functions satisfy
$\psi(\pi+\theta)=-\psi(\theta)$,
$\varphi(\pi+\theta)=-\varphi(\theta)$, $\theta\in[0,2\pi]$.
Assume $a>0$, with the other case being analogous.  By
Corollary~\ref{coro:1}, one has:
\begin{enumerate}
	\item If $\psi(\theta)< 0$ then $a\psi(\theta)-\varphi(\theta)\leq 0$. Moreover, as $\psi(\pi+\theta)=-\psi(\theta)>0$ then $\varphi(\theta+\pi)=-\varphi(\theta)\geq 0$. Hence $a\psi(\theta)\leq \varphi(\theta)\leq0$.
	\item If $\psi(\theta)>0$ then $\varphi(\theta)\geq 0$. Since $\psi(\theta+\pi)<0$ then $a\psi(\pi+\theta)-\varphi(\theta+\pi)\leq0$, i.e., $a \psi(\theta)-\varphi(\theta) \geq0$. Hence
$ a\psi(\theta)\geq \varphi(\theta)\geq0$.
	\end{enumerate}
In particular, $\psi(\theta)$ and $\varphi(\theta)$ have the same zeros and $(a\psi-\varphi)\varphi$ has positive definite sign, and the number of limit cycles is bounded (see~\cite{CLl,Pliss}), so we shall look for examples with more degrees of freedom.

\begin{exam}
Let us consider
the cubic system
\begin{equation}\label{systemplanar}
	\begin{split}
	x'&= a x-y+P_3(x,y), \\
	y'&=x+ay+Q_3(x,y),
	\end{split}
\end{equation}	
where
\begin{align*}
P_3(x,y)&=p_3x^3+p_2x^2y+p_1xy^2+p_0y^3,\\
Q_3(x,y)&=q_3x^3+q_2x^2y+q_1xy^2+q_0y^3,	
\end{align*}
with $p_0,\dots,p_3,q_0,\dots,q_3 \in\mathbb{R}$.	
Then
\begin{align*}
\varphi(\theta)&=p_3 \cos^4\theta+ (p_2+q_3) \cos^3\theta \sin\theta+ (p_1+q_2) \cos^2\theta \sin^2\theta \\
&\phantom +(p_0+q_1) \cos\theta \sin^3\theta +q_0 \sin^4\theta, \\
\psi(\theta)&= q_3 \cos^4\theta+ (q_2-p_3) \cos^3\theta \sin\theta+ (q_1-p_2) \cos^2\theta \sin^2\theta \\
&\phantom +(q_0-p_1) \cos\theta \sin^3\theta -p_0 \sin^4\theta.
\end{align*}

Note that $\varphi(\theta)=\varphi(\theta+\pi)$ and $\psi(\theta)=\psi(\theta+\pi)$. Then it suffices to check (i) and (ii) of Corollary \ref{coro:1} for $\theta \in [-\pi/2,\pi/2)$. Therefore, we can divide both $\psi$ and $\varphi$ by $\cos^4\theta$ to obtain a polynomial in $\tan \theta$. Note that this transformation does not affect the criteria.
The change of variables $t=\tan\theta$ transforms the interval $(-\pi/2,\pi/2)$ into $\mathbb{R}$. Define the polynomials
\[
P_\psi(t)=\frac{\psi(\arctan t)}{\cos^4\arctan t}=- p_0 t^4 +(q_0 - p_1)t^3 + (q_1-p_2)t^2 +(q_2-p_3)t + q_3,
\]
\[
P_\varphi(t)=\frac{\varphi(\arctan t)}{\cos^4\arctan t}= q_0 t^4 + (p_0+q_1)t^3+(p_1+q_2)t^2+(p_2+q_3)t+p_3.
\]
Then the signs of $\psi,\varphi,a\psi-\varphi$ in $\theta\in(-\pi/2,\pi/2)$ are the same as
the signs of $P_\psi,P_\varphi,a P_\psi-P_\varphi$ in $t=\tan \theta\in\mathbb{R}$.

Now choose
\begin{equation}\label{coefficients}\begin{split}
 a=1/2, \ &p_0=-1,\quad p_1=\frac{20731}{20000}, \quad p_2=\frac{-19}{1000}, \quad p_3=\frac{9}{10000},\\
& q_0=\frac{1}{2},\quad  q_1= \frac{2}{5},\quad  q_2=\frac{-17631}{20000}, \quad q_3=0.
\end{split}
\end{equation}
With this choice of coefficients, one has
\[
P_\psi(t)=\frac{(t-1 ) t (17649 + 9269 t + 20000 t^2)}{20000}.
\]
So $\psi(t)\leq 0$ for $t\in [0,1]$, and $\psi(t)>0$ elsewhere. Also,
\[
P_\varphi(t)=\frac{(10 t-9) (10t-1) (1 - 10 t + 50 t^2)}{10000},
\]
which has just two simple zeros, at $1/10$ and $9/10$, and is positive for $t\not\in (1/10,9/10)$,
so that (ii) of Corollary~\ref{coro:1} is satisfied. Now, to check that
(i) is also satisfied, it suffices to verify that $a P_\psi(t)-P_\varphi(t)\leq 0$ for $t\in[0,1]$.
This can be done by using Sturm's theorem or with any CAS with the appropiate command (for
instance, in Mathematica the command {\it CountRoots}).

Finally, we verify that system \eqref{systemplanar} with coefficients given by \eqref{coefficients} does not satisfy the hypotheses in Theorem 1, Theorem 3, and Corollary 4 in \cite{HL5}.

Define
\begin{equation*}\begin{split}
\omega_1(\theta)&=a \psi(\theta)-\varphi(\theta),\\
\omega_2(\theta)&= (n-1)(2a\psi(\theta)-\varphi(\theta))+ \psi'(\theta).
\end{split} \end{equation*}

We claim:
\begin{enumerate}
	\item There is no linear combination $\mu_1\omega_1(\theta)+ \mu_2\omega_2(\theta)$ where $\mu_1,\mu_2 \in\mathbb{R}$ which has definite sign.
	
	Define
		\[P_1(t)=\frac{\omega_1(\arctan t)}{\cos^4\arctan t},\quad P_2(t)=\frac{\omega_2(\arctan t)}{\cos^4\arctan t}.\]
		Recall that the signs of $P_1,P_2$ in $t\in\mathbb{R}$ are the signs of $\omega_1,\omega_2$ in $\theta=\arctan t\in(-\pi/2,\pi/2)$.
		
		Then $P_1,P_2$ are of degrees three and four, respectively, both with positive leading coefficient. Moreover, $P_1$ has a zero in $[-1/2,1/2]$, $t_0$,
		while $P_2(t)<0$ for all $t\in [-1/2,1/2]$. As $\omega_1$ has a change of sign, take $\mu_2\neq 0$ and consider $Q(t)=\mu_1 P_1(t)+\mu_2 P_2(t)$. The sign of $Q(t_0)$ is opposite to that of $Q(t)$ for $t$ close to $\infty$, so that $Q$ has no definite sign for any $\mu_1,\mu_2$.

	\item There are no $\nu_1,\nu_2\geq0$ such that $(a\nu_1)^2+ \nu_2^2 \neq0$ and $\omega_1(\theta)(\nu_1 a \psi(\theta)-\nu_2 \varphi(\theta))\leq0$.
	
	 The sign of $\omega_1(\theta)(\nu_1 a \psi(\theta)-\nu_2 \varphi(\theta))$ is the same as the sign
		of $P_1(t)(\nu_1 a P_\psi(t)-\nu_2 P_\varphi(t))$. In order to have definite sign,
		since $P_1$ is of degree three then $\nu_1 a P_\psi(t)-\nu_2 P_\varphi(t)$ must be an odd degree polynomial.
		But it is easy to check that this holds only for $\nu_1=\nu_2$, so that there is no $\nu_1,\nu_2>0$ such that
		$\omega_1(\theta)(\nu_1 a \psi(\theta)-\nu_2 \varphi(\theta))\leq 0$ unless $\nu_1=\nu_2=0$.

	\item Neither $\omega_1$ nor $\omega_2$  have  definite sign.
	
	This is a consequence of (1).
	
	\item Neither $-(n-1)\omega_1+\omega_2$ nor $-2(n-1)\omega_1+\omega_2$  have definite sign.
	
    This is a consequence of (1).
	
	\item Neither $a\omega_1\psi$ nor $ \omega_1\varphi$  have definite sign.
	
    Again the signs of $a\omega_1\psi$ and $\omega_1\varphi$ are
		the signs of $aP_1 P_\psi$ and $P_1P_\varphi$. Computing the roots of the polynomials,
		it is easy to check that they have no definite sign.
	
\end{enumerate}

\end{exam}

\end{document}